\documentclass[10pt, twoside, leqno]{article}
\usepackage{amssymb,amsthm}

\newtheorem{theorem}{Theorem}[section]
\newtheorem{corollary}[theorem]{Corollary}
\newtheorem{lemma}[theorem]{Lemma}

\newtheorem{question}[theorem]{Question}

\newtheorem{remark}[theorem]{Remark}
\newtheorem{example}[theorem]{Example}

\newtheorem{definition}[theorem]{Definition}

\newtheorem{observation}[theorem]{Observation}

\def\cprime{$'$}

\def\cc {{\mathfrak c}}

\def\RR {{\mathbb R}}

\def\sA {{\mathcal A}}
\def\sB {{\mathcal B}}
\def\sC {{\mathcal C}}
\def\sD {{\mathcal D}}
\def\sE {{\mathcal E}}

\def\sO {{\mathcal O}}
\def\sP {{\mathcal P}}

\def\sU {{\mathcal U}}
\def\sV {{\mathcal V}}
\def\sW {{\mathcal W}}

\def\min {\mathrm{min}}
\def\sup {\mathrm{sup}}

\def\< {{\langle}}
\def\> {{\rangle}}

\def\dot {\mathrm{dot}}
\def\st {\mathrm{st}}

\def\ot {\mathrm{ot}}

\begin{document}

\title{Regular $G_\delta$-diagonals and  some  upper  bounds  for
cardinality of topological spaces}

\author{I.\,S. Gotchev$^{1,2}$\\
Department of Mathematical Sciences,\\
Central Connecticut State University,\\ 
1615 Stanley Street,\\ 
New Britain, CT 06050, USA\\
E-mail: gotchevi@ccsu.edu
\and
M.\,G. Tkachenko$^2$\\
Departamento de Matem\'aticas,\\ 
Universidad Aut\'onoma Metropolitana,\\ 
Av. San Rafael Atlixco 186,\\ 
Col. Vicentina-Iztapalapa, 09340,\\
Mexico City, Mexico\\
E-mail: mich@xanum.uam.mx
\and
V.\,V. Tkachuk$^2$\\
Departamento de Matem\'aticas,\\ 
Universidad Aut\'onoma Metropolitana,\\ 
Av. San Rafael Atlixco 186,\\ 
Col. Vicentina-Iztapalapa, 09340,\\
Mexico City, Mexico\\
E-mail: vova@xanum.uam.mx}

\date{February 23, 2016}

\maketitle

\renewcommand{\thefootnote}{}


\footnote{\emph{Key  words  and  phrases}:   Cardinal   function,
regular  diagonal,  weakly  Lindel\"of  number, almost Lindel\"of
number, o-tightness, dense o-tightness, $\pi$-character, $\pi$-weight.}

\footnote{$^1$ The first  author  expresses  his gratitude to the
Mathematics   Department   at    the    Universidad    Aut\'onoma
Metropolitana,  Mexico  City,  Mexico,  for their hospitality and
support during his sabbatical visit of UAM in the spring semester
of 2015.}

\footnote{$^2$    Research    supported    by    CONACyT    grant
CB-2012-01-178103 (Mexico)}

\renewcommand{\thefootnote}{\arabic{footnote}}
\setcounter{footnote}{0}

\begin{abstract}
We  prove   that,   under   CH,   any   space   with   a  regular
$G_\delta$-diagonal  and  caliber  $\omega_1$  is  separable;   a
corollary  of  this  result  answers,  under  CH,  a  question of
Buzyakova. For any Urysohn space $X$, we establish the inequality
$|X|\le  wL(X)^{s\Delta_2(X)\cdot{\dot(X)}}$  which  represents a
generalization of a theorem of Basile, Bella, and Ridderbos.   We
also   show   that   if   $X$   is   a   Hausdorff   space,  then
$|X|\le(\pi\chi(X)\cdot    d(X))^{\ot(X)\cdot\psi_c(X)}$;    this
result       implies       \v{S}apirovski{\u\i}'s      inequality
$|X|\le\pi\chi(X)^{c(X)\cdot\psi(X)}$  which   only   holds   for
regular    spaces.    It    is    also    proved   that   $|X|\le
\pi\chi(X)^{\ot(X)\cdot\psi_c(X)\cdot aL_c(X)}$ for any Hausdorff
space $X$;  this  gives  one  more  generalization  of the famous
Arhangel$^\prime$skii's  inequality   $|X|\le 2^{\chi(X)\cdot  L(X)}$.
\end{abstract}

\section{Introduction}
If  a  space  has  a  $G_\delta$-diagonal, then there are notable
restrictions on its cardinal characteristics.  For example, if $X$ is a
regular Lindel\"of space with a $G_\delta$-diagonal, then it  has
a  weaker  second  countable topology. It is a result of Ginsburg
and Woods  \cite{GinWoo77},  that $|X|\le 2^{e(X)\cdot\Delta(X)}$
for any $T_1$-space $X$ and hence $|X| \le \mathfrak{c}$ whenever
a  $T_1$-space  $X$  with  a  $G_\delta$-diagonal  has  countable
extent. However, there are Tychonoff spaces  $X$  of  arbitrarily
large  cardinality  such  that $c(X)\cdot\Delta(X) = \omega$ (see
\cite{Sha84}). In particular, having $G_\delta$-diagonal and weak
Lindel\"of  property  does  not  restrict  the  cardinality  of   a
Tychonoff space.

The  situation  changes drastically if we assume that a space $X$
has a \emph{regular $G_\delta$-diagonal}, that is, there 
exists a countable family $\sU$ of open  neighborhoods of its 
diagonal $\Delta_X =
\{(x, x) : x  \in  X\}$  in  the  space  $X\times  X$  such  that
$\Delta_X  =\bigcap\{\overline{U}:U  \in \sU\}$. Buzyakova proved
in \cite{Buz06} that $|X|\le\mathfrak{c}$ whenever $X$ is a space
with a  regular  $G_\delta$-diagonal  such  that $c(X)\le\omega$.
Gotchev extended the result of  Buzyakova  establishing  in
\cite{Got15} that $|X|\le 2^{c(X)\cdot\overline{\Delta}(X)}$ for
any Urysohn space $X$.  It is worth mentioning that every space with 
a regular $G_\delta$-diagonal is Urysohn.

We recall that a space $X$ has a \emph{zero-set diagonal} if there 
is a continuous function $f\colon X\times X \rightarrow  \mathbb{R}$  
such that  $\Delta_X = f^{-1}(0)$. Clearly, having a zero-set diagonal
is   an   even   stronger   property   than   having   a  regular
$G_\delta$-diagonal.  In    \cite{Buz05}   Buzyakova
established that countable extent of $X \times X$ together with a
zero-set  diagonal  of  $X$ imply that $X$ is submetrizable and
asked whether any  space  with  a  zero-set  diagonal and caliber
$\omega_1$ also must be submetrizable. In this  paper we prove
that  the answer to this question is positive under the Continuum
Hypothesis (see Corollary~\ref{Cor:T}).

In \cite{BBR14} Basile, Bella and  Ridderbos proved that $|X| \le
wL(X)^{\pi\chi(X)}$  if  $X$  is  a  space  with  a  strong  rank
2-diagonal. Here, in Theorem~\ref{GT}, we generalize their result  
by  showing that  the  inequality $|X|\le wL(X)^{s\Delta_2(X)
\cdot{\dot(X)}}$ is true for any Urysohn space $X$.

In Corollary~\ref{CGT}  we   show   that   the   inequality
$|X|\le(\pi\chi(X)\cdot  d(X))^{\ot(X)\cdot\psi_c(X)}$  is   true
whenever  $X$  is  a  Hausdorff  space.  This  result implies 
immediately that for any Hausdorff space $X$ we have $|X|\le\pi
w(X)^{\ot(X)\cdot\psi_c(X)}$.   Together   with  Charlesworth's
inequality $d(X)\le\pi\chi(X)^{c(X)}$ which is valid for  regular
spaces, our result also implies \v{S}apirovski{\u\i}'s inequality
$|X|\le\pi\chi(X)^{c(X)\cdot\psi(X)}$  which  is known to be true
for any regular space $X$. We also generalize a result of Willard
and Dissanayake; they  proved  in  \cite{WilDis84}  that $|X| \le
2^{t(X)\cdot\psi_c(X)\cdot\pi\chi(X)\cdot   aL_c(X)}$   for   any
Hausdorff space $X$. We strengthen their  inequality  by  proving
in Theorem~\ref{TGT2} that  the  same  separation axiom in $X$ 
guarantees that $|X| \le\pi\chi(X)^{\ot(X)\cdot\psi_c(X)\cdot aL_c(X)}$.   
It is worth mentioning that the theorem of Willard and Dissanayake
generalizes the famous theorem of Arhangel'skii which states that
$|X|  \le  2^{\chi(X)\cdot  L(X)}$  whenever  $X$  is a Hausdorff
space.

\section{Notation and terminology}
Throughout this paper $\omega$ is (the cardinality of) the set of
all  non-negative  integers,  $\xi$,   $\eta$  and  $\alpha$  are
ordinals and $\kappa$,  $\tau$,  $\mu$  and  $\nu$  are  infinite
cardinals. The cardinality of the set $X$ is denoted by $|X|$ and
$\Delta_X  =  \{(x,x)\in X^2: x\in X\}$ is the \emph{diagonal} of
$X$. If $\sU$  is  a  family  of  subsets  of  $X$, $x\in X$, and
$G\subset X$ then $\st(G,\sU) = \bigcup\{U\in  \sU: U\cap  G\neq
\emptyset\}$.  When  $G=\{x\}$  we  write $\st(x,\sU)$ instead of
$\st(\{x\},\sU)$.          If          $n\in\omega$,         then
$\st^n(G,\sU)=\st(\st^{n-1}(G,\sU),\sU)$ and $\st^0(G,\sU)=G$.

All spaces are assumed to  be  topological  $T_1$-spaces.  For  a
subset $U$ of a space $X$ the closure of $U$ in $X$ is denoted by
$\overline{U}$.   As   usual,   $\chi(X)$  and  $\psi(X)$  denote
respectively the character  and  the  pseudocharacter of $X$. The
\emph{closed  pseudocharacter}  $\psi_c(X)$  (defined  only   for
Hausdorff  spaces $X$) is the smallest infinite cardinal $\kappa$
such  that   for   each   $x\in   X$,   there   is  a  collection
$\{V(\eta,x):\eta<\kappa\}$ of open  neighborhoods  of  $x$  such
that   $\bigcap_{\eta<\kappa}\overline{V}(\eta,x)   =   \{x\}$.    A
$\pi$-base  for  $X$ is a collection $\sV$ of non-empty open sets
in $X$ such that if  $U$  is  any  non-empty open set in $X$ then
there   exists   $V\in\sV$   such   that   $V\subset   U$.    The
\emph{$\pi$-weight}  of  $X$  is  $\pi w(X)=\min\{|\sV|:\sV$ is a
$\pi$-base for $X\}+\omega$.  A  family  $\sV$  of non-empty open
sets in $X$ is a local $\pi$-base at a point $x\in X$ if for every open
neighborhood  $U$  of $x$ there is $V\in \sV$ such that $V\subset
U$. The minimal  infinite  cardinal  $\kappa$  such that for each
$x\in X$ there is  a  collection  $\{V(\eta,x):\eta<\kappa\}$  of
non-empty open subsets of $X$ which is a local $\pi$-base for $x$
is  called  the  \emph{$\pi$-character}  of $X$ and is denoted by
$\pi\chi(X)$.

The \emph{Lindel\"of number} of $X$ is $L(X)=\min\{\kappa:$ every
open   cover   of   $X$    has    a   subcover   of   cardinality
$\leq\kappa\}+\omega$. The \emph{weak Lindel\"of number} of  $X$,
denoted  by  $wL(X)$,  is the smallest infinite cardinal $\kappa$
such  that  every  open  cover  of  $X$  has  a  subcollection of
cardinality  $\le\kappa$  whose  union  is  dense  in   $X$.   If
$wL(X)=\omega$  then  $X$ is called \emph{weakly Lindel\"of}. The
\emph{weak Lindel\"of degree of $X$  with respect to closed sets}
is denoted by $wL_c(X)$ and is defined as the  smallest  infinite
cardinal  $\kappa$  such  that for every closed subset $F$ of $X$
and every collection $\sV$ of  open  sets in $X$ that covers $F$,
there  is  a   subcollection   $\sV_0$   of   $\sV$   such   that
$|\sV_0|\le\kappa$   and  $F\subset\overline{\bigcup\sV_0}$.  The
\emph{almost Lindel\"of  number  of  $X$  with  respect to closed
sets} is denoted  by  $aL_c(X)$  and  is  the  smallest  infinite
cardinal  $\kappa$  such  that for every closed subset $F$ of $X$
and every collection $\sV$ of  open  sets in $X$ that covers $F$,
there  is  a   subcollection   $\sV_0$   of   $\sV$   such   that
$|\sV_0|\le\kappa$ and $\{\overline{V}:V\in\sV_0\}$ covers $F$. A
pairwise  disjoint  collection  of  non-empty open sets in $X$ is
called a \emph{cellular family}. The \emph{cellularity} of $X$ is
the  cardinal  $c(X)=\sup\{|\sU|:\sU$  is  a  cellular  family in
$X\}+\omega$.  We say that the \emph{o-tightness} of a space  $X$
does  not  exceed  $\kappa$,  or  $\ot(X)\le\kappa$, if for every
family $\sU$ of open subsets of  $X$ and for every point $x\in X$
with  $x\in\overline{\bigcup\sU}$  there   exists   a   subfamily
$\sV\subset\sU$      such      that      $|\sV|\le\kappa$     and
$x\in\overline{\bigcup\sV}$.  The \emph{tightness} at $x\in X$ is
$t(x,X)=\min\{\kappa:$   for    every    $Y\subseteq    X$   with
$x\in\overline{Y}$, there is $A\subset Y$ with $|A|\le\kappa$ and
$x\in\overline{A}\}$  and  the   \emph{tightness}   of   $X$   is
$t(X)=\sup\{t(x,X):x\in X\}+\omega$.

A space $X$ has a \emph{$G_\kappa$-diagonal} if there is a family
$\{U_\alpha:\alpha<\kappa\}$ of open sets  in  $X\times  X$  such
that     $\Delta_X=\bigcap_{\alpha<\kappa}     U_\alpha$;     if,
additionally,       $\Delta_X      =      \bigcap_{\alpha<\kappa}
\overline{U}_\alpha$    then    $X$     has    a    \emph{regular
$G_\kappa$-diagonal}. Clearly, when $\kappa=\omega$ then $X$  has
a        $G_\delta$-diagonal        (respectively,        regular
$G_\delta$-diagonal).  The \emph{diagonal degree} of $X$, denoted
$\Delta(X)$, is the smallest infinite cardinal $\kappa$ such that
$X$  has  a  $G_\kappa$-diagonal (hence $\Delta(X)=\omega$ if and
only if $X$ has a $G_\delta$-diagonal). It is worth noting that a
space $X$ has  a  regular  $G_\kappa$-diagonal  for some cardinal
$\kappa$ if and only if $X$ is a Urysohn  space.  For  a  Urysohn
space  $X$,  the  minimal  cardinal  $\kappa$ such that $X$ has a
regular $G_\kappa$-diagonal is  denoted by $\overline{\Delta}(X)$
and is called the \emph{regular diagonal degree of $X$}. Given  a
space  $X$  and  $n\in  \mathbb{N}$,  let  $s\Delta_n(X)$  be the
minimal  cardinal  $\kappa$  for  which  there  exists  a  family
$\{\sU_\alpha : \alpha <\kappa\}$ of open covers of $X$ such that
for  any   distinct   points   $x,y\in   X$   we   have  $y\notin
\overline{\st^n(x,\sU_\alpha)}$ for some  $\alpha  <  \kappa$.  A
space  $X$ has \emph{strong rank $n$-diagonal} if $s\Delta_n(X) =
\omega$.

\emph{Condensations} are  one-to-one  and  onto  continuous  mappings. A
space  $X$  is  \emph{submetrizable}  if  it  condenses  onto   a
metrizable     space,     or    equivalently,    $(X,\tau)$    is
\emph{submetrizable} if there  exists  a  topology $\tau'$ on $X$
such that $\tau'\subset \tau$ and $(X,\tau')$ is metrizable.

For definitions not given here and more information we refer  the
reader to \cite{E89}, \cite{Got15}, \cite{Gru84}, \cite{Hodel84},
\cite{Juh80} and \cite{Tka83}.

\section{On a question of Buzyakova}

Here  we  will  establish  that  caliber $\omega_1$ together with
regular $G_\delta$-diagonal is  equivalent  to separability under
the Continuum Hypothesis. We recall that $\omega_1$ is said to be
a \emph{caliber} of a space $X$ if any uncountable  family  $\sU$
of  non-empty  open  subsets  of $X$ has an uncountable subfamily
$\sU'$ such that  $\bigcap\sU'\ne\emptyset$.  It  is  easy to see
that every separable space has caliber $\omega_1$.

Buzyakova proved in \cite[Theorem~2.4]{Buz05}  that  a  separable
space  with a regular $G_\delta$-diagonal condenses onto a second
countable    Hausdorff    space.     She    also    asked    (see
\cite[Question~1.3]{Buz05})  whether  zero-set  diagonal  of  $X$
together with $\omega_1$ caliber imply that $X$ is submetrizable.
We give a positive answer to this question  under  the  Continuum
Hypothesis;  the same method gives a generalization, under CH, of
Theorem~2.4 of the paper \cite{Buz05}.

\begin{theorem}\label{TT}
Under the Continuum Hypothesis, if $X$  is a space with a regular
$G_\delta$-diagonal and caliber $\omega_1$ then $X$ is separable.
\end{theorem}

\begin{proof}
Since $X$ is a space with caliber $\omega_1$, we have $c(X)=\omega$. 
Buzyakova proved  in  \cite{Buz06}  that  if a space $X$ has a countable 
cellularity and a regular $G_\delta$-diagonal then $|X|\le 2^\omega$,   
so   we   have   $|X|\le 2^\omega=\omega_1$. Choose an  enumeration  
$\{x_\alpha :\alpha < \omega_1\}$   of   the   space   $X$   and   let   
$X_\alpha  = \overline{\{x_\beta  :  \beta  <  \alpha\}}$  for  all  
$\alpha < \omega_1$. It is immediate  that 
$U_\alpha = X\setminus X_\alpha$
is an open subset of $X$ for every $\alpha$ and that  the  family
$\sU  =  \{U_\alpha  :  \alpha  <  \omega_1\}$  is decreasing and
point-countable. Since $\omega_1$  is  a  caliber  of  $X$, it is
impossible that all elements of $\sU$ be non-empty and  therefore
$X_\alpha = X$ for some $\alpha<\omega_1$, so $X$ is separable.
\end{proof}

The  following  corollary  answers  Question~3.1  of \cite{Buz05}
under CH.

\begin{corollary}\label{Cor:T}
Under the Continuum Hypothesis,  if  $X$  has a zero-set diagonal
and caliber $\omega_1$ then X is submetrizable. 
\end{corollary}

\begin{proof}
Since zero-set diagonal implies regular  $G_\delta$-diagonal,  we
can  apply Theorem~\ref{TT} to see that $X$ must be separable, so
it   is   submetrizable   by    a    theorem   of   Martin   (see
\cite[Theorem~2.1]{Mar75}). 
\end{proof}

The  corollary  that   follows   generalizes   Theorem~2.4   of
\cite{Buz05} under CH.

\begin{corollary}
Under   the   Continuum   Hypothesis,   if   $X$  has  a  regular
$G_\delta$-diagonal and  caliber  $\omega_1$  then  $X$ condenses
onto a second countable Hausdorff space. 
\end{corollary}

\begin{proof}
Apply Theorem~\ref{TT} to see that $X$  is  separable  and  hence
Theorem~2.4  of  \cite{Buz05}  is applicable to conclude that $X$
condenses onto a second countable Hausdorff space. 
\end{proof}

\section{Bounds given  by the weak Lindel\"of number}

In this section we will show  that the weak Lindel\"of number and
dense o-tightness give an  upper  bound  on  the  cardinality  of
spaces with strong rank 2-diagonals; observe that such spaces are
automatically  Urysohn.  Basile,  Bella  and  Ridderbos proved in
\cite{BBR14} that if $X$ is a space with a strong rank 2-diagonal
then   $|X|\le   wL(X)^{\pi\chi(X)}.$   Theorem   \ref{GT}  below
generalizes their result. For its statement we need the following
definition.

\begin{definition}
We will say that the \emph{dense o-tightness}  of  $X$  does  not
exceed $\kappa$, or $\dot(X)\le\kappa$, if for every family $\sU$
of  open subsets of $X$ whose union is dense in $X$ and for every
point $x\in X$ there exists a subfamily $\sV\subset\sU$ such that
$|\sV|\le\kappa$         and         $x\in\overline{\bigcup\sV}$.
\end{definition}

The observation below follows immediately from the definitions.

\begin{lemma}
The inequalities  $\dot(X)\le\ot(X)$,  $\dot(X)\le\pi\chi(X)$ and
$\dot(X)\leq c(X)$ are valid for every space $X$. 
\end{lemma}

\begin{theorem}\label{GT}
For    every    Urysohn    space    $X$    we    have     $|X|\le
wL(X)^{s\Delta_2(X)\cdot{\dot(X)}}$. 
\end{theorem}

\begin{proof}
Assume  that  $wL(X)\le\lambda$,  $\dot(X)\le\mu$  and, for some
infinite cardinal $\kappa$, let  $\{\sU_\eta: \eta<\kappa\}$ be a
family witnessing the  inequality  $s\Delta_2(X)\le\kappa$.   For
each    $\eta    <    \kappa$,    we    can    fix    a    family
$\mathcal{D}_\eta\subset\mathcal{U}_\eta$        such        that
$|\mathcal{D}_\eta|\le\lambda$       and       $\overline{\bigcup
\mathcal{D}_\eta}=X$.                  The                 family
$\sD=\bigcup\{\sD_\eta:\eta<\kappa\}$    has    cardinality   not
exceeding $\lambda\cdot\kappa$.

It  follows  from  $\dot(X)\le\mu$  that  for  every $x\in X$ and
$\eta<\kappa$ we can find a family $\sV'_\eta(x) \subset \mathcal
D_\eta$  such  that   $x\in  \overline{\bigcup\sV'_\eta(x)}$  and
$|\sV'_\eta(x)|\le\mu$.    The   cardinality   of   the    family
$\sV_\eta(x)=\{V\in                            \sV'_\eta(x): V\cap
\st(x,\sU_\eta)\ne\emptyset\}$  does  not exceed $\mu$ either and
$x\in  \overline{\bigcup\sV_\eta(x)}$  for   any   $x\in  X$  and
$\eta<\kappa$.

Letting $F(x)(\eta) = \sV_\eta(x)$ for any point $x\in X$ and any
ordinal      $\eta<\kappa$       we       define       a      map
$F\colon X\rightarrow[\sD]^{\le\mu}$. To see that $F$ is injective 
take any  pair  $x,y$  of  distinct  points  of  $X$.   There   exists
$\eta<\kappa$  such  that  $y\notin\overline{\st^2(x,\sU_\eta)}$.
Then    $F(x)(\eta)    =    \sV_\eta(x)$    which    shows   that
$x\in\overline{\bigcup\sV_\eta(x)}$                           and
$\bigcup\sV_\eta(x)\subset\st^2(x,\sU_\eta)$                   so
$y\notin\overline{\bigcup\sV_\eta(x)}$. Therefore $F(y)(\eta) \ne
\sV_\eta(x)=F(x)(\eta)$ and hence $F(x)\ne F(y)$. This shows that
$F$     is     an     injective     map     and     consequently,
$|X|\le|[\sD]^{\le\mu}|  \le(\lambda\cdot\kappa)^\mu  \le(\lambda
\cdot \lambda^\kappa)^\mu\le\lambda^{\kappa\cdot\mu}$    as
promised. 
\end{proof}

\begin{corollary}
If $X$ is a  space  with  a  strong rank 2-diagonal, then $|X|\le
wL(X)^{\dot(X)}$. 
\end{corollary}

\begin{corollary}
For  every  space  $X$  of countable dense o-tightness and strong
rank 2-diagonal we have $|X|\le wL(X)^{\omega}$. 
\end{corollary}

\begin{corollary}
If  $X$  is  a  weakly   Lindel\"of  space  with  a  strong  rank
2-diagonal, then $|X|\le 2^{\dot(X)}$. 
\end{corollary}

\begin{corollary}
Assume that $X$ is a weakly Lindel\"of space of countable  dense
o-tightness and strong rank 2-diagonal. Then $|X|\le 2^{\omega}$.
\end{corollary}

\begin{corollary}\label{CTk}
If  $X$  is a Urysohn space, then $$|X|\le wL(X)^{\pi\chi(X)\cdot
s\Delta_2(X)}.$$ 
\end{corollary}

\begin{corollary}
For  every  Urysohn  space  $X$  we  have  $$|X|\le  2^{c(X)\cdot
s\Delta_2(X)}.$$ 
\end{corollary}

Gotchev    established     in     \cite{Got15}    that    $|X|\le
wL(X)^{\chi(X)\cdot  \overline{\Delta}(X)}$  for  any  Urysohn
space $X$ so it  would  be  interesting  to  see  whether  it  is
possible  to  prove  the following simultaneous generalization of
his result and Corollary~\ref{CTk}.

\begin{question}
Is it  true  that  the  inequality $|X|\le wL(X)^{\pi\chi(X)\cdot
\overline{\Delta}(X)}$   holds   for    all    Urysohn    spaces?
\end{question}

We will  show next that Theorem \ref{GT} indeed improves the 
result of Basile, Bella and Ridderbos mentioned in the beginning of this Section.

\begin{example} \rm 
It  is  well  known  that  the  Cantor  cube  $K=\{0,1\}^\cc$  is
separable; let $X$  be  a  countable  dense  subspace of $K$. The
space  $X$   being   submetrizable   and   countable,   we   have
$s\Delta_2(X)=      L(X)     =wL(X)=     \dot(X)=\omega$     while
$\pi\chi(X)=\pi\chi(K)=\cc$. Therefore the formula

\medskip
$wL(X)^{s\Delta_2(X)\cdot \dot(X)} =\omega^\omega=\cc <2^\cc =
wL(X)^{s\Delta_2(X)\cdot \pi\chi(X)}$

\medskip\noindent
witnesses that Theorem~\ref{GT}  is  strictly  stronger than the
result of Basile, Bella and Ridderbos.
\end{example}

The following example shows  that  there  is even a compact space
$X$ such that $\dot(X)<\min\{\ot(X),\pi\chi(X)\}$.

\begin{example}\label{Ex:AB} \rm
For every  infinite  cardinal  $\tau$,  there  exists  a  compact
Hausdorff          space          $X$          such          that
$\dot(X)=\tau<\min{\{\ot(X),\pi\chi(X)\}}$. 
\end{example}

\begin{proof}
Let  $Y=D^\kappa$  be  the  Cantor  cube,  where $\kappa=\tau^+$.
Denote by $Z$ the space  $\kappa+1$  of ordinal numbers less than
or equal to $\kappa$ endowed with the order  topology.  

Fix a point $y_0\in Y$ and let
$X$  be  the  quotient  space  of the topological sum $Y\oplus Z$
obtained by identifying $y_0$ with the
point $\kappa\in Z$. It is clear that $Y$, $Z$, and $X$ are
compact Hausdorff spaces. The space $X$ contains closed copies of
$Y$ and $Z$, where $Z'=Z\setminus\{\kappa\}$ accumulates  at  the
point  $y_0$  \lq\lq{outside\rq\rq}  of  $Y$. We denote the point
$\{y_0,\{\kappa\}\}$   of    $X$    by    $x_0$.    Notice   that
$Y\setminus\{y_0\}$ and $Z'$ are disjoint  open  subsets  of  $X$
once  $Y$ and $Z$ are identified with the corresponding subspaces
of $X$. 

Let  $\mathcal{U}$  be  a family of open subsets of $X$ such that
$\bigcup\mathcal{U}$ is dense  in $X$.  Then $\mathcal{V}=\{U\cap
Y: U\in\mathcal{U}\}$ is a family of open sets in $Y$ whose union
is dense in $Y$.  Since  $Y$  has  countable  cellularity,  there
exists a countable subfamily $\mathcal{V}'$ of $\mathcal{V}$ such
that  $\bigcup\mathcal{V}'$  is  dense  in  $Y$. Take a countable
subfamily    $\mathcal{U}'$    of    $\mathcal{U}$    such   that
$\mathcal{V}'=\{U\cap      Y:      U\in\mathcal{U}'\}$.      Then
$Y\subset\overline{\bigcup\mathcal{U}'}$  and,   in   particular,
$x_0\in\overline{\bigcup\mathcal{U}'}$.  If  $x\in X\setminus Y$,
that is $x\in Z'$, then we  use the inequality $\chi(Z')\leq\tau$ to
choose a subset  $A\subset  Z'\cap\bigcup\mathcal{U}$  such  that
$x\in\overline{A}$  and  $|A|\leq\tau$.  Then we take a subfamily
$\mathcal{W}$       of        $\mathcal{U}$       such       that
$A\subset\bigcup\mathcal{W}$ and $|\mathcal{W}|\leq\tau$.  It  is
clear       that        $x\in\overline{A}\subset\overline{\bigcup
\mathcal{W}}$. This implies that $\dot(X)\leq\tau$.

Since $Z'$ is open  in  $X$  and  $x_0$ compactifies $Z'$, we see
that $\ot(X)=\ot(Z)=\kappa$. The same  argument  along  with  the
equality         $\pi\chi(y_0,Y)=\kappa$        imply        that
$\pi\chi(X)=\pi\chi(Z)=\kappa$.       It       follows       that
$\dot(X)=\tau<\kappa=\min\{\ot(X),\pi\chi(X)\}$. 
\end{proof}

\begin{remark}
{\rm We do  not  know  whether  the  difference  between $\dot(X)$ 
and $\min\{\ot(X),\pi\chi(X)\}$  can  be  arbitrarily  large  for   a
compact  space  $X$. However, the  difference   in  question for 
non-compact spaces  can  be arbitrary large. 
To see this one can strengthen the  topology  of
the   ordinal   space   $Z=\kappa+1$  in  Example~\ref{Ex:AB}  by
declaring the points of $Z'$ isolated  and taking the sets of the
form $Z\setminus A$, with $A\subset  Z'$  and  $|A|\leq\tau$,  as
basic  open neighborhoods of the point $\{\kappa\}$. Let $Z^*$ be
the   resulting    space.    Repeating    the   construction   in
Example~\ref{Ex:AB} applied to the topological  sum  of  $Y$  and
$Z^*$,    we   obtain   a   quotient   space   $X^*$   satisfying
$\dot(X^*)=\omega$ and $\min{\{\ot(X^*),\pi\chi(X^*)\}}=\kappa$.}
\end{remark}

\section{Bounds on cardinality involving o-tightness}

In 1969 Arhangel{\cprime}skii proved  that the inequality $|X|\le
2^{\chi(X)\cdot L(X)}$ is valid for  every  Hausdorff  space  $X$
(\cite{Arh69}). In 1972 \v{S}apirovski{\u\i} improved this result
by   showing   that   $|X|\le   2^{t(X)\cdot\psi(X)\cdot   L(X)}$
(\cite{Sap72}).    Later   Willard   and   Dissanayake   improved
Arhangel{\cprime}skii's  inequality   by   showing  that  $|X|\le
2^{t(X)\cdot\psi_c(X)\cdot\pi\chi(X)\cdot               aL_c(X)}$
(\cite{WilDis84}). Then Bella  and  Cammaroto  noticed  that  the
cardinal  function  $\pi\chi(X)$ could be omitted and showed that
$|X|\le  2^{t(X)\cdot\psi_c(X)\cdot  aL_c(X)}$ (\cite{BelCam88}).
Their result is also a generalization  of  \v{S}apirovski{\u\i}'s
inequality.  The following theorem gives another strengthening of
the  theorem  of  Willard  and  Dissanayake;  therefore  it  also
generalizes  Arhangel{\cprime}skii's  inequality.   It  is  worth
noting that $\ot(X)\le t(X)$ and $\ot(X)\le c(X)$ for  any  space
$X$.

\begin{theorem}\label{TGT2}
If  $X$  is  a  Hausdorff space, then \begin{equation}\label{Eq4}
|X|\le       \pi\chi(X)^{\ot(X)\cdot\psi_c(X)\cdot      aL_c(X)}.
\end{equation} 
\end{theorem}

\begin{proof}
Let     $\pi\chi(X)=\tau$     and      $\ot(X)\cdot\psi_c(X)\cdot
aL_c(X)=\kappa$.   For  each $x\in X$ choose a $\pi$-base $\sU_x$
at the point $x$ and  a  family  $\sV_x$ of open neighborhoods of
$x$   such   that   $|\sU_x|\le\tau$,   $|\sV_x|\le\kappa$    and
$\bigcap\{\overline{V}:V\in\sV_x\}=\{x\}$.

For  every subset $A$ of $X$ let $\sV_A=\bigcup\{\sV_x:x\in A\}$,
$\sU_A=\bigcup\{\sU_x:x\in  A\}$  and  $\sU_A(V)=\{U:U\subset  V,
U\in\sU_x, x\in A\cap V\}$ whenever $V$ is an open subset of $X$.
Let also  $\mathfrak{W}_A=\{\sW:|\sW|\le\kappa, \sW\subset \sV_A,
X\setminus\bigcup\{\overline{V}:V\in\sW\}\ne\emptyset\}$,
$\mathfrak{O}_A=\{\sO:|\sO|\le\kappa,  \sO\subset  \sU_A\}$   and
$\mathfrak{P}_A=\{\sP:|\sP|\le\kappa,  \sP\subset \mathfrak{O}_A,
(\bigcap\{\overline{\bigcup\sO}:\sO\in\sP\})\setminus
A\ne\emptyset\}$. For each $\sW\in\mathfrak{W}_A$ we pick a point
$p_A(\sW)\in  X\setminus\bigcup\{\overline{V}:V\in\sW\}$  and for
each  $\sP\in\mathfrak{P}_A$  we  pick   a   point   $q_A(\sP)\in
(\bigcap\{\overline{\bigcup\sO}:\sO\in\sP\})\setminus A$.

Now let $F_0=\{z\}$ where $z$ is an  arbitrary  point  in  $X$.  
Recursively  we construct a family $\{F_\eta:\eta<\kappa^+\}$ of
subsets of $X$ as follows:   
\begin{itemize}  
   \item[(i)]  If $\eta<\kappa^+$ is  a      limit   ordinal   then
                   $F_\eta={\bigcup\{F_\xi:\xi<\eta\}}$; 
   \item[(ii)]  If $\eta=\xi+1$ then    $F_\eta=    {F_\xi   \cup   
                  \{p_{F_\xi}(\sW)   :\sW   \in
                  \mathfrak{W}_{F_\xi}\}   \cup  \{q_{F_\xi}(\sP):
                  \sP\in\mathfrak{P}_{F_\xi}\}}$. 
\end{itemize}

Observe first  that  $|F_0|\le  \tau^\kappa$;  we  will  prove by
transfinite induction that  $|F_\eta|\le  \tau^\kappa$  for  each
$\eta<\kappa^+$.  Assume  that   for  each  $\xi<\eta$,  where
$\eta<\kappa^+$, we have $|F_\xi|\le \tau^\kappa$. If $\eta$ is a
limit              ordinal              then              clearly
$|\bigcup\{F_\xi:\xi<\eta\}|\le\tau^\kappa$.        Now       let
$\eta=\alpha+1$.  Since  $|F_\alpha|\le  \tau^\kappa$,  we   have
$|\{p_{F_\alpha}(\sW):\sW\in\mathfrak{W}_{F_\alpha}\}|\le
(\kappa\cdot\tau^\kappa)^\kappa=\tau^\kappa$                  and
$|\{q_{F_\alpha}(\sP):\sP\in\mathfrak{P}_{F_\alpha}\}|\le
((\tau\cdot\tau^\kappa)^\kappa)^\kappa=\tau^\kappa$.    Then   it
follows from (ii) that $|F_\eta|\le\tau^\kappa$.

It is  clear  that  the set $F={\bigcup\{F_\eta:\eta<\kappa^+\}}$
has cardinality not exceeding $\tau^\kappa$. We will first  prove
that  $F$  is  a  closed  set  in  $X$;  this  fact  will be used
afterwards to show that $F=X$.

Suppose  that  $F$   is   not   closed.   Then   there  is  $x\in
\overline{F}\setminus  F$.  Let  $V\in\sV_x$.  If  $V'$  is   any
neighborhood of $x$ then $V'\cap V$ is a non-empty open subset of
$X$ and therefore there is $y\in F\cap V'\cap V$ and $U\in \sU_y$
such  that  $U\subset  V'\cap  V$.  This  shows  that  for  every
$V\in\sV_x$ we have $x\in \overline{\bigcup \sU_F(V)}$. Therefore
there  exists $\sO_V\subset\sU_F(V)$ such that $|\sO_V|\le\kappa$
and         $x\in         \overline{\bigcup\sO_V}$.         Since
$\bigcup\sO_V\subset\bigcup\sU_F(V)\subset     V$,   we  have
$\overline{\bigcup\sO_V}\subset\overline{V}$.           Therefore
$\{x\}=\bigcap\{\overline{\bigcup\sO_V}:V\in\sV_x\}$.   For   any
$V\in\sV_x$ and $U\in\sO_V$ choose a point $y=y(V,U)\in F$ such
that  $U\in\sU_y$.  It  is  clear that the cardinality of the set
$D=\{y(V,U):V\in\sV_x$   and   $U\in\sO_V\}$   does   not  exceed
$\kappa$. Therefore there is $\xi<\kappa^+$ such  that  $D\subset
F_\xi$.  Hence,  $\sO_V\in  \mathfrak{O}_{F_\xi}$  whenever $V\in
\sV_x$.     If     $\sP=\{\sO_V:V\in\sV_x\}$     then     $\sP\in
\mathfrak{P}_{F_\xi}$ and  clearly  $x=q_{F_\xi}(\sP)$. Therefore
$x\in F_{\xi+1}\subset F$, which is a  contradiction.  The  proof
that $F$ is closed is complete.

Now  suppose that there is $x\in X\setminus F$. For each $y\in F$
let  $V_y\in\sV_y$  be  such  that  $x\notin\overline{V_y}$. Then
$\sV=\{V_y:y\in F\}$ is an open cover of $F$. Thus, there  exists
$\sV'\subset\sV$      such     that     $|\sV'|\le\kappa$     and
$F\subset\bigcup\{\overline{V_y}:V_y\in\sV'\}$.  Clearly $x\notin
\bigcup\{\overline{V_y}:V_y\in\sV'\}$. Let  $C=\{y:V_y\in\sV'\}$.
Then  $|C|\le\kappa$  and  therefore there is $\xi<\kappa^+$ such
that             $C\subset              F_\xi$.             Since
$X\setminus\bigcup\{\overline{V_y}:V_y\in\sV'\}\ne\emptyset$,  we
have $p_{F_\xi}(\sV')\notin \bigcup\{\overline{V_y}:V_y\in\sV'\}$
and  at  the  same  time   $p_{F_\xi}(\sV')\in   F_{\xi+1}\subset
F\subset        \bigcup\{\overline{V_y}:V_y\in\sV'\}$.       This
contradiction completes the proof. 
\end{proof}

\begin{corollary}
For   every    Hausdorff    space    $X$    we    have   
$$|X|\le \pi\chi(X)^{\ot(X)\cdot\psi(X)\cdot L(X)}.$$ 
\end{corollary}

\begin{proof}
The claim follows directly from Theorem \ref{TGT2} and  the  fact
that  if  $X$ is a Hausdorff space then $\psi_c(X)\le\psi(X)\cdot
L(X)$ (see \cite[2.9(c)]{Juh80}). 
\end{proof}

\begin{corollary}
If  $X$  is  a  Urysohn  space,  then \begin{equation}\label{Eq6}
|X|\le       \pi\chi(X)^{\ot(X)\cdot\psi(X)\cdot        aL_c(X)}.
\end{equation} 
\end{corollary}

\begin{proof}
The inequality follows directly from Theorem~\ref{TGT2} and the fact
that if $X$  is  a  Urysohn  space then $\psi_c(X)\le\psi(X)\cdot
aL_c(X)$ (see \cite[Lemma~2.1]{Hodel06}). 
\end{proof}

\begin{corollary}\label{CGT4}
If   $X$   is   a    Hausdorff    space    with    $\pi\chi(X)\le
2^{\ot(X)\cdot\psi_c(X)\cdot      aL_c(X)}$      then     $|X|\le
2^{\ot(X)\cdot\psi_c(X)\cdot aL_c(X)}$. 
\end{corollary}

We  present  next an example  which  shows that Theorem~\ref{TGT2} 
indeed improves the theorem of Willard and Dissanayake.   
Recall that for a Tychonoff space $X$, the expression $C_p(X)$ stands  
for  the  set  of  all real-valued   continuous   functions  on  $X$  endowed  
with the pointwise convergence topology.   For the  basic facts about the
spaces $C_p(X)$ we refer the reader to the book \cite{Tk11}.

\begin{example}
\rm
Apply Theorem~2.1 of \cite{OT96} to see that there exists a 
Tychonoff space $Z$ with the following properties:

{
\smallskip
\noindent\hangafter=1 \hangindent=.3 in\rlap{(i)}\hskip.3in
$Z=D\cup A\cup\{p\}$ where  $D$  is  a countable dense set of
isolated points of $Z$ while the sets $A,D$ and $\{p\}$ are disjoint;

\noindent\hangafter=1 \hangindent=.3 in\rlap{(ii)}\hskip.3in
the  subspace  $A\cup\{p\}$  is  compact and $p$ is its 
unique non-isolated point;

\noindent\hangafter=1 \hangindent=.3 in\rlap{(iii)}\hskip.3in
$|A|=\cc$ and the space $C_p(Z)$ is Lindel\"of.

}

\smallskip
Observe first  that  the  space  $Z$  is  separable and therefore
$C_p(Z)$  has  a  weaker  second  countable  topology   and,   in
particular, $\psi(C_p(Z))=\omega$ (see \cite[Problem~173]{Tk11}).
Take  a set $A_0\subset A$ such that $|A_0|=|A\setminus A_0|=\cc$
and let $u\in \RR^Z$ be the function such that $u(z)=0$ for all 
$z\in  A_0$ and $u(z)=1$ whenever $z\in Z\setminus A_0$; 
it is clear that $u\in \RR^Z \setminus C_p(Z)$.
We will prove even more, namely, that

\smallskip\noindent
$(*)$ \ if  $Q\subset  C_p(Z)$   and  $|Q|<\cc$,  then  $u  \notin
\overline Q$ (the closure is taken in $\RR^Z$).

\smallskip
To verify $(*)$ note first that for every $f\in Q$ there exists a
countable set $A_f\subset A$ such that $f(z)=f(p)$ for any  $z\in
A\setminus  A_f$.  The  cardinality of the set $A'= \bigcup\{A_f:
f\in Q\}$ is strictly  less  than  $\cc$  so we can find 
points  $z_0\in  A_0\setminus  A'$  and  $z_1\in (A\setminus A_0)
\setminus  A'$.   It  is  straightforward  that  $U=\{g\in \RR^Z:
|g(z_0)|<  \frac13$  and  $|g(z_1)-1|<  \frac13\}$  is  an   open
neighborhood of $u$  in  $\RR^Z$  such  that $U\cap Q=\emptyset$;
this settles $(*)$.

Take  a countable dense subset $E$ in the space $\RR^Z$ such that
$u\in  E$.   Then  $X=C_p(Z)\cup  E$  is  a  separable Lindel\"of
subspace of $\RR^Z$ and it follows from  $(*)$  that  $t(X)=\cc$.
Besides,     $\psi(C_p(Z))=\omega$     easily     implies    that
$\psi(X)=\omega$.  The space  $X$  is  Tychonoff  and  Lindel\"of so  
we have $\psi_c(X)= \psi(X) =\omega$ and $aL_c(X)= L(X)=\omega$.  
Since $X\subset\RR^Z$ and $|Z|\le\cc$, we have $w(X)\le\cc$. Hence
$\pi\chi(X)\le\cc$. Also, since $X$ is separable,  we have $\ot(X)=\omega$. 
Therefore the formula
$$\pi\chi(X)^{\ot(X)\cdot \psi_c(X)\cdot aL_c(X)} 
\le\cc^\omega=\cc <2^{\cc} =
2^{t(X)\cdot \psi_c(X)\cdot \pi\chi(X)\cdot  aL_c(X)}$$
witnesses that Theorem~\ref{TGT2}  is  strictly  stronger than the
inequality of Willard and Dissanayake.
\end{example}

The result in Theorem~\ref{TGT2}  should  be  compared  with
\v{S}apirovskii's inequality

\begin{equation}\label{Eq5}
|X|\le\pi\chi(X)^{c(X)\cdot\psi(X)}.  
\end{equation}   

\noindent
proved  in
\cite{Sap74} for  regular  spaces.  If  $X$  is  the  Kat\v{e}tov
extension $\kappa{\hskip1pt}\omega$ of the set $\omega$ with 
the discrete topology,     then    $X$  is       Urysohn    and
$\pi\chi(X)=c(X)=\psi(X)=\omega$  while  $|X|=2^{\mathfrak{c}}$;
this  shows  that \v{S}apirovskii's inequality is not true if we drop 
the regularity of $X$. The inequality (\ref{Eq6}) shows one of the
possible ways to find a statement analogous to  (\ref{Eq5})  that
holds for Urysohn spaces.

Our next result will allow us to obtain a direct strengthening of
(\ref{Eq5}) for Hausdorff spaces.

\begin{theorem}\label{TG1}
If  $X$   is   a   Hausdorff   space   and   $A\subset   X$  then
$$|\overline{A}|\le
(\pi\chi(X)\cdot|A|)^{\ot(X)\cdot\psi_c(X)}.$$ 
\end{theorem}

\begin{proof}
Let  $\psi_c(X)=\nu$,  $\ot(X)=\mu$ and $\pi{\chi}(X)=\tau$.  For
each $x\in X$ choose  a  local $\pi$-base $\sU_x$ at the point
$x$  with  $|\sU_x|\leq  \tau$  and  a  family  $\sV_x$  of  open
neighborhoods of $x$ such that $\{x\}=\bigcap\{ \overline V: V\in
\sV_x\}$ and $|\sV_x|\leq \nu$; let $\sU=\bigcup\{\sU_x: x\in A\}$.

For  any  $x\in  \overline{A}$ and $V\in\sV_x$ the cardinality of
the  family  $\sU_A(V)=\{U: U\subset  V,  U\in  \sU_y,  y\in A\cap
V\}\subset \sU$ does not exceed $\tau\cdot|A|$. If  $V'$  is  any
neighborhood of $x$ then $V'\cap V$ is a non-empty open subset of
$X$ and therefore there is $y\in A\cap V'\cap V$ and $U\in \sU_y$
such that  $U\subset  V'\cap  V$.   This  shows  that  for  every
$V\in\sV_x$  we have $x\in \overline{\bigcup\sU_A(V)}$. Therefore
for every  $V\in\sV_x$  there  exists $\sO_V\subset\sU_A(V)$ such
that $|\sO_V|\le\mu$ and $x\in \overline{\bigcup\sO_V}$.

Since $\bigcup\sO_V \subset \bigcup\sU_A(V)\subset  V$,  we  have
the inclusion  $\overline{\bigcup\sO_V}  \subset  \overline{V}$. 
Therefore $\{x\}=\bigcap\{\overline{\bigcup \sO_V}: V\in\sV_x\}$.
Observing that $|\sU|\leq \tau\cdot |A|$ and every $\sO_V$  is  a
subfamily of $\sU$ of cardinality $\leq \mu$ we conclude that the
cardinality  of  the  collection $\sO=\{\sO_V:  V\in \sV_x,\ x\in
A\}$ does not exceed ${(\tau\cdot|A|)}^{\mu}$. Hence there are at
most    $({{(\tau\cdot|A|)}^\mu)}^\nu$-many    intersections   of
$\nu$-many elements of $\sO$.  We already saw that every point of
$\overline A$ \ is an intersection of $\nu$-many elements of $\sO$
so  $|\overline{A}|\le{(\tau\cdot|A|)}^{\mu\cdot\nu}$. 
\end{proof}

\begin{corollary}\label{CGT2}
If   $X$   is   a   Hausdorff   space,   then  $|\overline{A}|\le
\pi\chi(X)^{\ot(X)\cdot\psi_c(X)}$  whenever  $A\subset  X$ and
$|A|\le\pi\chi(X)^{\ot(X)\cdot\psi_c(X)}$. 
\end{corollary}

\begin{corollary}\label{CGT1}
Assume  that  $X$ is a Hausdorff space and $A\subset X$. Then
$$|\overline{A}|\le(\pi\chi(\overline{A})\cdot
|A|)^{\ot(\overline{A})\cdot\psi_c(\overline{A})}.$$
\end{corollary}

\begin{corollary}\label{CGT}
For every Hausdorff  space  $X$  we have $$|X|\le(\pi\chi(X)\cdot
d(X))^{\ot(X)\cdot\psi_c(X)}.$$ 
\end{corollary}

Observe  that  Corollary~\ref{CGT}   implies   \v{S}apirovskii's
inequality  (\ref{Eq5})  because  for  every  space  $X$  we have
$\ot(X)\le      c(X)$      while      $\psi_c(X)=\psi(X)$     and
$d(X)\le\pi\chi(X)^{c(X)}$ whenever $X$ is a regular  space  (see
\cite{Cha77}).

\begin{corollary}\label{CGT3}
If  $X$ is a Hausdorff space, then $$|X|\le d(X)^{\pi\chi(X)\cdot
\ot(X)\cdot\psi_c(X)}.$$ 
\end{corollary}

\begin{theorem}\label{TGT1}
Let   $X$    be    a    Hausdorff    space.    Then   $$|X|\le\pi
w(X)^{\ot(X)\cdot\psi_c(X)}.$$ 
\end{theorem}

\begin{proof}
It follows directly from Corollary~\ref{CGT} and  the  fact  that
for every topological space $X$ we have $\pi w(X)=\pi\chi(X)\cdot
d(X)$ (see \cite[3.8(b)]{Hodel84}). 
\end{proof}

\begin{observation}\label{OV}
{\rm  It  is  natural  to  ask  whether  the  result of Bella and
Cammaroto is stronger than our inequality (\ref{Eq4}). This would
happen if  $2^{t(X)}\le  \pi\chi(X)^{\ot(X)}$  for  any Hausdorff
space $X$. However,  this  is  false  even  for  compact  spaces.
Indeed,  if  X  is the Tychonoff cube $[0, 1]^\mathfrak{c}$, then
$\pi\chi(X) = t(X) = \mathfrak{c}$ and $\ot(X) \le c(X) = \omega$
so  $\pi\chi(X)^{\ot(X)}  =   \mathfrak{c}   <  2^\mathfrak{c}  =
2^{t(X)}$.

If the inequality $\pi\chi(X)^{\ot(X)}\le 2^{t(X)}$ were true for
all Hausdorff spaces, then Theorem~\ref{TGT2}  would  imply  the
inequality  of Bella and Cammaroto. However, this inequality does
not hold either:  Take  $X$  to  be the $\Sigma$-product of
$2^\mathfrak{c}$-many real lines and observe that $t(X) = \omega$
while $\pi\chi(X) = 2^\mathfrak{c}$ so $2^{t(X)} <  \pi\chi(X)  =
\pi\chi(X)^{\ot(X)}$.

To  see  that  Corollary~\ref{CGT4}  gives  new  information, it
suffices to prove that there exists a Hausdorff  space  $X$  such
that  $\pi\chi(X)  >  2^{\ot(X)\cdot\psi_c(X)\cdot aL_c(X)}$.  We
will show that there  are  models  of  ZFC  in which such a space
exists and is even normal.} \
\end{observation}

\begin{theorem}\label{TV}
There is a model of ZFC in which we can find a regular (and hence
normal) hereditarily Lindel\"of space $X$ such that $\pi  w(X)  >
\mathfrak{c}$. 
\end{theorem}

\begin{proof}
Hajnal and Juh\'asz proved in \cite{HajJuh73} that there exists a
model  of  ZFC  in  which  GCH holds and we can find a set $E$ of
cardinality $\omega_1$ and a family  $\sA$ of subsets of $E$ such
that $|\sA| = \omega_2$ and
\begin{itemize}
   \item[(a)] if $k \in \mathbb{N}$ and $\{A_{nm} : n \in \omega,  1
    \le  m  \le  k\}$  is  a subfamily of $\sA$ such that $A_{nm} \ne
    A_{n'm'}$   whenever   $(n,m)\ne    (n',m')$,    then   the   set
   $E\setminus\bigcup_{n\in\omega} B_n$ is countable  provided  that
   every  $B_n$  is  the intersection $B_n^1\cap\ldots\cap B^k_n$ in
   which the set $B^i_n$  is  either $A_{ni}$ or $E\setminus A_{ni}$
   for each $i \le k$.
   \item[(b)]  for  any  $x  \in  E$  and   countable   $B   \subset
   E\setminus\{x\}$,  there  exists  $A  \in \sA$ such that $x \in A
   \subset E\setminus B$.
\end{itemize}

Let $Y$ be the set $E$  with the topology generated by the family
$\{A, E\setminus A : A \in \sA\}$ as a subbase. It was proved  in
\cite{HajJuh73}   that   $Y$   is   a   regular  zero-dimensional
hereditarily Lindel\"of space. Denote by  $\sC$ the family of all
open countable subsets of $Y$; it is immediate that the set $G =
\bigcup\sC$ is  countable.  We  claim  that  the  space  $X  =  Y
\setminus  G$  is  as  promised; of course, we only need to prove
that $\pi w(X) = \omega_2$.

Striving for a contradiction, assume that there is a family $\sB$
of non-empty open subsets of  $X$  such that $|\sB| \le \omega_1$
and $\sB$ is a $\pi$-base in $X$. Observe first that  it  follows
from   our   choice  of  $X$  that  all  elements  of  $\sB$  are
uncountable. If infinitely many elements of $\sA$ are  countable,
then we can find a faithfully indexed subfamily 
$\sA'  =  \{A_n  :   n \in\omega\}$ of the family $\sA$ whose all
elements are countable. However, this implies that
$E\setminus\bigcup\sA'$ is uncountable which
is a contradiction  with  (a).  
Therefore at most finitely  many  elements of $\sA$ are countable
and hence we can find a family $\sE \subset \sA$ such that $|\sE|
= \omega_2$ and $A \cap X \ne \emptyset$ for all $A \in \sE$.

Observe that $A \cap X$ is a non-empty open subset of $X$ for any
$A \in  \sE$  so  it  follows  from  the  fact  that  $\sB$  is a
$\pi$-base in $X$ and $|\sB| \le \omega_1$ that there is  $B  \in
\sB$  such  that  the  set  $\{A  \in  \sE  :  B \subset A\}$ has
cardinality $\omega_2$. In particular,  we  can find a faithfully
indexed family $\{A_n : n \in\omega\} \subset\sE$  such  that  $B
\subset  A_n$  for  all $n \in \omega$. As a consequence, $B \cap
(E\setminus A_n) = \emptyset$ for  each  $n \in \omega$ and hence
the set  $E\setminus\bigcup\{E\setminus  A_n  :  n  \in  \omega\}
\supset B$ is uncountable which is a contradiction with (a). This
contradiction proves that $\pi w(X) = \omega_2$. 
\end{proof}

\begin{corollary}\label{CV}   
For  the  space  $X$  from  Theorem~\ref{TV} we have  
$\pi\chi(X)  =  2^\mathfrak{c} > \mathfrak{c} =
2^{c(X)\cdot \psi(X)\cdot  L(X)}$  and  therefore  $\pi\chi(X)  >
2^{\ot(X)\cdot\psi_c(X)\cdot aL_c(X)}$. Thus, Corollary~\ref{CGT4}  
gives   new   information,   at  least  consistently.
\end{corollary}

\begin{proof} If $\pi\chi(X)\le\mathfrak{c} = \omega_1$, then  it
follows  from $|X| = \omega_1$ that $\pi w(X) \le \omega_1$ which
is a contradiction with Theorem \ref{TV}. Therefore $\pi\chi(X) =
\omega_2 = 2^\mathfrak{c}$ so all  that  is  left is to note that
$c(X) = \psi(X) = L(X) =  \omega$  because  $X$  is  hereditarily
Lindel\"of.  Finally, observe that it follows from the regularity
of  $X$  that  $\psi_c(X)  =  \psi(X)$  and  $aL_c(X)  \le L(X)$.
\end{proof}

The following question seems to be  interesting  because if it has
an affirmative answer, the respective statement will be a
simultaneous  generalization  of  Theorem~\ref{TGT2} and   
\v{S}apirovski{\u\i}'s  inequality  (\ref{Eq5})  in  the  class  of
$T_3$-spaces.

\begin{question}\label{Q3}
Is the  inequality  $$|X|\le  \pi\chi(X)^{\ot(X)\cdot\psi(X)\cdot
wL_c(X)}$$ true for every regular space $X$? 
\end{question}

\subsection*{Acknowledgements}
We express our gratitude to the referee for carefully reading our
paper and a stimulating critique.

\end{document}